\documentclass[11pt,a4paper]{article}
\usepackage{color}
\usepackage{graphicx}
\usepackage{amsfonts}
\usepackage{extarrows}
\usepackage{amsmath,amsthm,amssymb,color}
\usepackage[backref=page]{hyperref}
\usepackage{eepic}
\usepackage{lineno}
\usepackage{enumerate}	
\usepackage{paralist}

\usepackage{algorithm}
\usepackage{algorithmicx}
\usepackage{algpseudocode}

\usepackage{tikz}
\newtheorem{theorem}{Theorem}[section]
\newtheorem{proposition}[theorem]{Proposition}
\newtheorem{lemma}[theorem]{Lemma}
\newtheorem{corollary}[theorem]{Corollary}

\theoremstyle{definition}
\newtheorem{definition}[theorem]{Definition}
\newtheorem{claim}{Claim}

\newtheorem{conjecture}{Conjecture}[section]
\newtheorem{question}{Question}

\begin{document}
%\pagewiselinenumbers
\title{\bf Integer colorings with no rainbow $k$-term arithmetic progression}
\author{Hao Lin \thanks{School of Mathematics, Shandong University,
Jinan, China, Email: \texttt{lhao17@mail.sdu.edu.cn}.}
\quad Guanghui Wang\thanks{School of Mathematics, Shandong University,
Jinan, China, Email: \texttt{ghwang@sdu.edu.cn}, supported by Natural Science Foundation of China (11871311) and Young Taishan Scholars Program of Shandong Province (201909001).}
%\quad Wenling Zhou\thanks{{\bf Correspondence:} School of Mathematics, Shandong University,Jinan, China, Email: \texttt{gracezhou@mail.sdu.edu.cn}.}
\quad Wenling Zhou\thanks{School of Mathematics, Shandong University, Jinan, China, and  Laboratoire Interdisciplinaire des Sciences du Num\'{e}rique, Universit\'{e} Paris-Saclay, France. Email: \texttt{gracezhou@mail.sdu.edu.cn}.}}
\date{}
\maketitle

\begin{center}
\begin{minipage}{130mm}
\small\noindent{\bf Abstract:}
In this paper, we study the rainbow Erd\H{o}s-Rothschild problem with respect to $k$-term arithmetic progressions. For a set of positive integers $S \subseteq [n]$,
an $r$-coloring of $S$ is \emph{rainbow $k$-AP-free} if it contains no rainbow $k$-term arithmetic progression. Let $g_{r,k}(S)$ denote the number of rainbow $k$-AP-free $r$-colorings of $S$. For sufficiently large $n$ and fixed integers $r\ge k\ge 3$, we show that $g_{r,k}(S)<g_{r,k}([n])$ for any proper subset $S\subset [n]$. Further, we prove that $\lim_{n\to \infty}g_{r,k}([n])/(k-1)^n= \binom{r}{k-1}$. Our result is asymptotically best possible and implies that, almost all rainbow $k$-AP-free $r$-colorings of $[n]$ use only $k-1$ colors.


\end{minipage}
\end{center}

%---------------------
\section{Introduction}
Let $k$ and $n$ be positive integers and $[n]:= \{1, 2, \dots, n\}$.
Given a subset $S\subseteq [n]$, a \textit{$k$-term arithmetic progression} ($k$-AP for short) in $S$ is a set of distinct elements of the form
$a, a+d, a+2d, \dots, a+(k-1)d$ with
difference $d\in [n]$.
For a positive integer $r$ and a set $S$, an \textit{$r$-coloring} of $S$ is a mapping $\chi : S \to [r]$, which assigns one color to each element of $S$. We say such an
$r$-coloring is \textit{exact} if $\chi$ is surjective.
A $k$-AP in $S$ is called \textit{rainbow} (or \textit{monochromatic}) if all of its elements receive distinct (or same) colors under an $r$-coloring of $S$.

Arithmetic progressions play an important role in number theory.
The well-known van der Waerden's Theorem~\cite{1927Beweis} in Ramsey theory states that, for every $r$ and $k$ and sufficiently
large $n$, every $r$-coloring of $[n]$ contains a monochromatic $k$-AP.
%, which has played a significant role in the development of additive combinatorics and combinatorial ergodic theory.
A strengthening of this result was conjectured by Erd\H{o}s and Tur\'{a}n~\cite{E-T-1936} in 1936, and finally proved by Szemer\'{e}di~\cite{SE-k-AP} in 1975.

\begin{theorem}{\rm (Szemer\'{e}di's Theorem~\cite{SE-k-AP})}\label{Sz-thm}
For a positive integer $k$ and a real number $\alpha >0$, there exists a positive integer $sz(k, \alpha)$ such that for all $n\ge sz(k, \alpha)$, every subset $S \subseteq [n]$ with $|S| \ge \alpha n$ contains a $k$-AP.
\end{theorem}

The anti-Ramsey (more precisely, anti-van der Waerden) properties of arithmetic progressions have received great attention as well (see \cite{anti-r-AP, anti2003} and references therein).
Given a positive integer $k$ and a set $S$, the \textit{anti-van der Waerden number} $aw(S, k)$ is the smallest $r$ such that every exact $r$-coloring of $S$ contains a rainbow $k$-AP.
In~\cite{anti-r-AP}, Butler et al. proved that for sufficiently large $n$, $aw([n], 3)=\Theta(\log n)$ and $aw([n], k)=n^{1-o(1)}$ for $k\ge 4$.
Note that this definition of $aw(S, k)$ implies that
there are always $r$-colorings of $S$ with no rainbow $k$-AP for all $r<aw(S, k)$.
An $r$-coloring of $S$ is called \textit{rainbow $k$-AP-free} if it contains no rainbow $k$-AP.

Following a recent trend of working on problems about counting certain colorings and analyzing extremal structure of them, for positive integers $r$, $k$ and set $S\subseteq [n]$, we are interested in the problem of counting the number, $g_{r,k}(S)$, of
rainbow $k$-AP-free $r$-colorings of $S$.
In addition, we naturally want to know, whether the behavior of $g_{r,k}([n])$ is similar to that of $aw([n], k)$, which is different for $k=3$ and  for $k\ge 4$.

The related problem of counting colorings of discrete objects with certain properties, was initiated by Erd\H{o}s-Rothschild~\cite{E-R1974} who
were interested in $n$-vertex graphs that admit the maximum number of $r$-edge-colorings
without a given monochromatic subgraph, which has motivated a number of results (see, e.g.,~\cite{m-cliques2004,other-coloer2006,m-subgraphs2014,m-cliques2017}).
For the Erd\H{o}s-Rothschild problems on other discrete structures,
we refer to~\cite{linear-hypergraphs} for $k$-uniform hypergraphs, to~\cite{in-set-families2018} for intersecting set families, to~\cite{chain2019} for $k$-chains and to~\cite{sum-set-2018,sum-set2021} for sum-free sets.
Meanwhile, the \emph{rainbow Erd\H{o}s-Rothschild problem} has received considerable attention, which studies the maximum number of $r$-edge-colorings of $n$-vertex graphs without a given \textit{rainbow} subgraph whose edges are assigned distinct colors.
For example, in~\cite{rainbow-triangle2019, rainbow-triangle2020, rainbow-triangle2017}, they studied the number of $r$-edge-colorings of graphs avoiding rainbow triangles. For further results along this line of research, we refer to~\cite{r-E-R-problem} and the references therein.
Recently, the rainbow Erd\H{o}s-Rothschild problem has also been extended to other discrete structures. Cheng et al.~\cite{2020Integer} studied the number of rainbow sum-free colorings of integers.
Li, Broersma and Wang~\cite{2021Integer} considered the problem for $3$-AP.

In this paper, we extend the results of 3-AP in~\cite{2021Integer} and study rainbow Erd\H{o}s-Rothschild problem with respect to $k$-AP.
For $r\in[k-1]$, we trivially have $g_{r,k}(S)=r^{|S|}$ for all $S\subseteq [n]$.
Hence $[n]$ is the unique set admitting the maximum number of rainbow $k$-AP-free $r$-colorings for all $k\ge 3$ and $1<r <k$.
For $r \ge k$, the research on $g_{r,k}(S)$ requires substantially more work.
In addition, note that a \textit{sum} in $S$ is a solution of the equation $x+y=z$, and a $3$-AP in $S$ can be regarded as a solution of $x+y=2z$.
Therefore, when we intend to count the number of rainbow $3$-AP-free (or sum-free) colorings of $S$, we can find some required $3$-APs (or sums) by using the ordered pairs in $S$ (see~\cite{2020Integer,2021Integer}).
Unfortunately, for $k\ge 4$, there is no such linear equation to help us find $k$-APs that are needed.
Therefore, we will introduce some new ideas in the proof.

For all $r \ge k$ and $S\subseteq [n]$, let $f^{<k}(r,S)$ denote the number of $r$-coloring of $S$ using at most $k-1$ colors. Then we have a lower bound on $g_{r,k}(S)$ as follows.

\begin{theorem}\label{l-bound}
For all $r \ge k\ge 2$ and $S\subseteq [n]$ with $|S|\ge k$,
\[
g_{r,k}(S)\ge f^{<k}(r,S)= \sum_{t=1}^{k-1}\bigg(t^{|S|}\sum_{j=t}^{k-1}\binom{r}{j}\binom{j}{j-t}(-1)^{j-t}\bigg).
\]
\end{theorem}

For the upper bound, we first estimate the number of rainbow $k$-AP-free $r$-colorings of dense sets.

\begin{theorem}\label{up-bound-dense}
For all integers $r \ge k\ge 3$ and any real number $0\leq \xi \leq r^{-3}/2$, there exists an integer $n_0$ such that the following holds for $n\ge n_0$. If $S\subseteq [n]$ with $|S|= (1-\xi)n$, then
\[
g_{r,k}(S)\le \binom{r}{k-1}(k-1)^{|S|}+(k-1)^{-\frac{1-\log(k-2)\cdot \log^{-1}(k-1)}{8k^3\log n}n} (k-1)^n.
\]
\end{theorem}

Now we turn to the extremal configurations of rainbow $k$-AP-free $r$-colorings for $r\ge k$.

\begin{theorem}\label{r-extremal}
For all integers $r \ge k\ge 3$, there exists an integer $N$ such that for all $n\ge N$ and any proper subset $S\subset [n]$, we have $g_{r,k}(S)<g_{r,k}([n])$.
Moreover, we have
\[
\lim_{n\to \infty} \frac{g_{r,k}([n])}{(k-1)^n}=\binom{r}{k-1}.
\]

\end{theorem}

\noindent \textit{Remark}. Theorem~\ref{r-extremal} is trivial for $k=2$, since $g_{r,2}(S)=r$ for all $S\subseteq [n]$ and $r\ge 1$. For $k=3$, it was also proved in~\cite{2021Integer}. Combining the previous discussion for $r\in [k-1]$, for all integers $r\ge 2$, $k\ge 3$ and sufficiently large $n$, $[n]$ is always the unique subset admitting the maximum number of rainbow $k$-AP-free $r$-colorings among all subsets of $[n]$. In particular, almost all rainbow $k$-AP-free $r$-colorings of $[n]$ use at most $k-1$ colors.
\bigskip

\noindent {\bf Organization.} The rest of this paper is organized as follows. In the next section, we will state and prove some useful results on $k$-APs.
Moreover, we will derive a container theorem (see Corollary~\ref{r-k-ap-container}) for rainbow $k$-AP-free $r$-colorings using the hypergraph container method.
In Section~\ref{sec-l-bound}, we give a combinatorial proof of Theorem~\ref{l-bound}. In Section~\ref{pf-u-bound}, we prove some necessary lemmas and give the proof of Theorem~\ref{up-bound-dense}. In Section~\ref{prove-r-extremal}, we will prove Theorem~\ref{r-extremal}. We close the paper with some concluding remarks in Section~\ref{remarks}.

%%%%%%%%%%%%%%%%%%%%%%%%%%%%%%%%%%%%%%%%%%%%%%%%%%%%%%%%%%%%%%%%%%%%%%%%%%%%%%%%%%%%%%%%%%
\section{Notation and preliminaries}\label{sec-Preliminaries}
\noindent {\bf Notation.}
In this paper all logarithms are on base 2.
We write $[a, b]$ for the interval $\{a, a + 1, \dots, b\}$. Given a set $S\subseteq [n]$ and an integer $a\in S$, denote by $\Gamma_k(S)$ the number of $k$-APs in $S$ and $\Gamma_k(a^+,S)$ the number of $k$-APs with $a$ as the first term in $S$.
Throughout the paper, we fix that $r$ and $k$ are both positive integers, and we do not consider the $k$-APs with common difference 0.

\subsection {Basic properties of {\it k}-APs in [{\it n}]}
In order to prove Theorem~\ref{up-bound-dense}, we need a supersaturation result for $k$-APs. For arithmetic progressions, the first supersaturation result proved by Varnavides~\cite{varnavides1959certain} is that any subset of $[n]$ of size $\Omega(n)$ contains $\Omega(n^2)$ $3$-APs.
For $k\ge 3$, Shkredov~\cite{shkredov2006szemeredi} formulated a sharpening of Szemer\'{e}di's Theorem due to Varnavides~\cite{varnavides1959certain}.

\begin{theorem}{\rm (\cite[Theorem 18]{shkredov2006szemeredi})}\label{supersaturation}
For every $k\ge 3$ and a real number $\alpha >0$, there exists a positive integer $v(k, \alpha)$ such that for all $n\ge v(k, \alpha)$ and $S \subseteq [n]$, if
$|S| \ge \alpha n$, then $\Gamma_k(S)\ge \frac{\alpha^2}{16v^3(k, \alpha)}n^2$.
\end{theorem}

In addition, we have the following proposition.

\begin{proposition}\label{[n]}
For all integers $n\ge k\ge 2$ with $n\equiv k' \pmod {k-1}$, we have
\[
\Gamma_k([n])=\frac{n^2}{2(k-1)}-\frac{n}{2}+\frac{k'(k-1-k')}{2(k-1)}.
\]
\end{proposition}

\begin{proof}
Set $q=\frac{n-k'}{k-1}$. Let us consider a partition of $[n]$ into $V_0,V_1, \dots, V_q$ with $V_0=[1,k']$ (if $k'=0$ then $V_0=\emptyset$) and $V_i=[k'+(i-1)(k-1)+1,k'+i(k-1)]$ for $i\in[q]$. Then for each integer $a\in V_i$, $\Gamma_k(a^+,[n])=q-i$ for $0\le i\le q$. Therefore, we have
\[
\Gamma_k([n])= \sum_{a\in [n]}\Gamma_k(a^+,[n])=\sum_{i=0}^{q}|V_i|(q-i)
             =\frac{n^2}{2(k-1)}-\frac{n}{2}+\frac{k'(k-1-k')}{2(k-1)}.
\]
\end{proof}

\subsection{The hypergraph container mathod}
We will use the method of hypergraph containers in the proof of Theorem~\ref{up-bound-dense}.
This powerful method was developed independently by Balogh-Morris-Samotij~\cite{B-container} and Saxton-Thomason~\cite{S-container}.

A \textit{$k$-uniform hypergraph} $\mathcal{H}=(V, E)$ consists of a vertex set $V$ and an edge set $E$, in which every edge is a set of $k$ vertices in $V$. An \textit{independent} set in $\mathcal{H}$ is a set of vertices inducing no edge in $E$.
Many classical theorems in combinatorics can be phrased as statements about
independent sets in given hypergraphs. For example, Szemer\'{e}di's Theorem can be expressed as for $V(\mathcal{H})=[n]$ and $E(\mathcal{H})$ consisting of all $k$-APs in $[n]$, the size of the maximum independent set of $\mathcal{H}$ is $o(n)$.

Let $\mathcal{H}=(V, E)$ be a $k$-uniform hypergraph with average degree $d:=k|E|/|V|$. For every
$X\subseteq V(\mathcal{H})$, let $\mathcal{H}[X]$ be the subhypergraph of $\mathcal{H}$ induced by $X$, and let $\deg(X)=|\{e\in E(\mathcal{H}) :  X\subseteq e\}|$ denote the \textit{co-degree} of $X$.
For every $ j\in [k]$, denote by $\Delta_j(\mathcal{H})$ the $j$th maximum co-degree of $\mathcal{H}$, that is, $\Delta_j(\mathcal{H})=\max\{\deg(X): \ X\subseteq V(\mathcal{H}),\ |X|=j\}$. When the underlying hypergraph is clear, we simply write it as $\Delta_j$.
For any real number $0<\tau<1$, the \textit{co-degree function} $\Delta(\mathcal{H},\tau)$ is defined as
\[
\Delta(\mathcal{H},\tau)=2^{\binom{k}{2}-1}\sum_{j=2}^k2^{-\binom{j-1}{2}}\frac {\Delta_j}{d\tau^{j-1}}.
\]

We will use the following version of the hypergraph container theorem.

\begin{theorem}{\rm (\cite[Corollary 3.6]{S-container})}
\label{baloghcontainer}
Let $\mathcal{H}$ be a $k$-uniform hypergraph on $n$ vertices. Let $0<\varepsilon, \tau< 1/2$.
Suppose that $\tau<1/(200k!^2k)$ and $\Delta(\mathcal{H},\tau)\leq \varepsilon/(12k!)$. Then there exists $c=c(k)\leq 1000k!^3k$ and a collection $\mathcal{C}$ of vertex subsets such that
\begin{compactenum}[\rm (i)]
  \item every independent set in $\mathcal{H}$ is a subset of some $I\in \mathcal{C}$;\smallskip
  \item for every $I\in \mathcal{C}$, $|E(\mathcal{H}[I])|\leq \varepsilon |E(\mathcal{H})|$;\smallskip
  \item $\log|\mathcal{C}|\leq cn\tau \log(1/\varepsilon)\log(1/\tau)$.
\end{compactenum}
\end{theorem}

A key concept in applying the hypergraph container method to such coloring problems is the notion of \textit{template}, which was first formally introduced in~\cite{template2019}, although the concept had already appeared in~\cite{other-coloer2006} under the name of ``2-colored multigraphs".

\begin{definition}(Template, palette, subtemplate, rainbow $k$-AP template)\label{def-template}
\end{definition}
\begin{itemize}
  \item An \textit{$r$-template} of order $n$ is a function $P: [n] \to 2^{[r]}$, associating with each element $x \in [n]$ a list of colors $P(x) \subseteq [r]$. We refer to this set $P(x)$ as the \textit{palette} available at $x$.
  \item Let $P_1$, $P_2$ be two $r$-templates of order $n$. We say that $P_1$ is a \textit{subtemplate} of $P_2$ (i.e. $P_1 \subseteq P_2$) if $P_1(x) \subseteq P_2(x)$  for every $x \in [n]$.
  \item For an $r$-template $P$ of order $n$, we say that $P$ is a \textit{rainbow $k$-AP $r$-template} if there exists a $k$-AP $K=\{a_1, \dots, a_k\}$ such that $|P(a_i)|=1$ for $i\in [k]$, $P(x)=\emptyset$ for $x\in [n]\setminus K$, and $P(a_1), \dots, P(a_k)$ are pairwise distinct.
\end{itemize}

Note that for any $S \subseteq [n]$, an $r$-coloring of $S$ can be considered as an $r$-template $P$ of order $n$ with $|P(x)|=1$ for $x\in S$ and $P(x)=\emptyset$ for $x\in [n]\setminus S$. For an $r$-template $P$ of order $n$, write $R_k(P)$ for the number of subtemplate of $P$ that are rainbow $k$-AP templates. We say that $P$ is a \textit{rainbow $k$-AP-free} $r$-template if $R_k(P)=0$.
Given positive integers $n\ge r\ge k \ge 3$, consider the \textit{$k$-AP-hypergraph} $\mathcal{H}_k$
encoding the vertex set $[n]\times [r]$ and the edge set of $\mathcal{H}_k$ consists of all $k$-tuples $\{(a_1, c_1), \dots, (a_k, c_k)\}$ such that $\{a_1, \dots, a_k\}$ forms a
$k$-AP in $[n]$ and $c_1,\dots, c_k$ are $k$ pairwise distinct
colors in $[r]$.
Using the $k$-AP-hypergraph $\mathcal{H}_k$, we obtain the following adaption of Theorem~\ref{baloghcontainer} to rainbow $k$-AP-free $r$-template.

\begin{corollary}\label{r-k-ap-container}
For integers $r\ge k\ge 3$ and sufficiently large $n$, there exists a constant $c=c(r,k)$ and a collection $\mathcal{C}$ of $r$-templates of order $n$ such that
\begin{compactenum}[\rm (i)]
  \item every rainbow $k$-AP-free $r$-template of order $n$ is a subtemplate of some $P\in \mathcal{C}$;\smallskip
  \item for every $P\in \mathcal{C}$, $R_k(P)<k^{-1}n^{2-1/k}$;\smallskip
  \item $\log|\mathcal{C}|\leq cn^{\frac{k-1}{k}} \log^2 n$.
\end{compactenum}
\end{corollary}
\begin{proof}
Consider the $k$-AP-hypergraph $\mathcal{H}_k=(V,E)$. Clearly, $|E|=\frac{r!}{(r-k)!}\Gamma_k([n])$.
By Proposition~\ref{[n]}, we have $|E|>\frac{r!n^2}{2k(r-k)!}$.
By the definitions of $d$ and $\Delta_j$, it is easy to check that
\begin{equation}
d=\frac{k|E|}{|V|}> \frac{(r-1)!}{2(r-k)!}n, \ \ \
\Delta_k=1, \ \ \
\Delta_j\le \Delta_2\le \binom{k}{2}\frac{(r-2)!}{(r-k)!}< k^2r^{k-2}
\ \ \text{for} \ 2\le j \le k-1.
\end{equation}
Now we apply Theorem~\ref{baloghcontainer} on $\mathcal{H}_k$. Let $\varepsilon=\frac{(r-k)!}{r!}n^{-\frac{1}{k}}$ and
$\tau=(2^k\cdot12k!\cdot r)^{\frac{1}{k-1}}n^{-\frac{1}{k}}$.
%Then we have $\tau<r^{-2k}$ for large $n$.
Let $\alpha_j=2^{-\binom{j-1}{2}} \tau^{1-j}$ for $2\le j\le k$.
Then we have
\begin{equation}
\frac{\alpha_j}{\alpha_{j+1}}=2^{j-1}\tau <2^k\tau<1 \ \ \text{for} \ 2\le j \le k-1 \ \
\text{and} \ \
\frac{\alpha_{k-1}}{\alpha_{k}} k^3r^{k-2}=k^3r^{k-2}2^{k-2}\tau<1.
\end{equation}
Next, let us bound the function $\Delta(\mathcal{H}_k,\tau)$ as follows
\begin{align*}
\Delta(\mathcal{H}_k,\tau)  &=  2^{\binom{k}{2}-1}d^{-1}\sum_{j=2}^k\alpha_j{\Delta_j}
\overset{(1)}{\le}  2^{\binom{k}{2}-1}d^{-1} \big(k^2r^{k-2}\sum_{j=2}^{k-1}\alpha_j +\alpha_k\big)\\
& \overset{(2)}{\le}  2^{\binom{k}{2}-1}d^{-1} \big((k-2)k^2r^{k-2}\alpha_{k-1}+\alpha_k\big)
\overset{(2)}{\le} 2^{\binom{k}{2}}d^{-1}\alpha_{k}
\overset{(1)}{\le} \frac{\varepsilon}{12k!}.
\end{align*}
Hence, there exists a collection $\mathcal{C}$ of vertex subsets satisfying properties (i)-(iii) of Theorem~\ref{baloghcontainer}. Observe that every vertex subset of $\mathcal{H}_k$ can be regarded as an $r$-template of order $n$, and every independent set in $\mathcal{H}_k$ can be regarded as a rainbow $k$-AP-free $r$-template of order $n$. Therefore, $\mathcal{C}$ is a desired collection of $r$-templates of order $n$. Moreover, by Proposition~\ref{[n]} and the property (ii) of Theorem~\ref{baloghcontainer}, for every $P\in \mathcal{C}$, $R_k(P)\leq \varepsilon |E|= n^{-1/k}\Gamma_k([n])<k^{-1}n^{2-1/k}$.
\end{proof}
For convenience, we define good $r$-template with respect to subset $S\subseteq [n]$.
\begin{definition}[Good $r$-template]
For every subset $S\subseteq [n]$, an $r$-template $P$ of order $n$ is a \textit{good $r$-template} for $S$ if it satisfies the following properties:\smallskip
\begin{compactenum}[(i)]
  \item For every $x \in S$, $|P(x)|\geq 1$;\smallskip
  \item $R_k(P)<k^{-1}n^{2-1/k}$.
\end{compactenum}
\end{definition}

%%%%%%%%%%%%%%%%%%%%%%%%%%%%%%%%%%%%%%%%%%%%%%%%%%%%%%%%%%%%%%%%%%%%%%%%%%%%%%%%%%%%%%
\section{Proof of Theorem~\ref{l-bound}}\label{sec-l-bound}
Theorem~\ref{l-bound} can be obtained directly, using the following Proposition~\ref{f(r,S)}. Given a positive integer $r$ and a set $S\subseteq [n]$ with $|S|\ge r$, let $f(r,S)$ denote the number of exact $r$-colorings of $S$.

\begin{proposition}\label{f(r,S)}
For a positive integer $r$  and a set $S$ with $|S|\ge r$,
\[
f(r,S)=\sum_{i=0}^{r-1}\binom{r}{i}(r-i)^{|S|}(-1)^{i}.
\]
\end{proposition}

\begin{proof}
We use induction on $r$. For $r=1$ this is clearly true, so we assume that the equality holds for the number of summands up to $r$, and
prove it for $r+1$.
By the definition of $f(r,S)$, we have
\begin{equation}\label{exact-coloring}
f(r+1,S)=(r+1)^{|S|}-\sum_{j=1}^{r}\binom{r+1}{j}f(j,S).
\end{equation}
Fix an integer $0\le i\le r-1$. By the induction hypothesis, the coefficient of the $(r-i)^{|S|}$ term of the equation~(\ref{exact-coloring}) is
\[
-\sum_{j=0}^{i}\binom{r+1}{r-j}\binom{r-j}{i-j}(-1)^{i-j}
=\frac{(r+1)!}{(r-i)!(i+1)!}\sum_{j=0}^{i}\frac{(i+1)!}{(i-j)!(j+1)!}(-1)^{i+1-j}.
\]
By the Binomial Theorem, the coefficient of the $(r-i)^{|S|}$ term is reduced to
\[
\frac{(r+1)!}{(r-i)!\cdot(i+1)!}\cdot(-1)^{i+1}=\binom{r+1}{i+1}\cdot(-1)^{i+1}.
\]
Therefore,
\[
f(r+1,S)=(r+1)^{|S|}+\sum_{i=0}^{r-1}\binom{r+1}{i+1}(-1)^{i+1}(r-i)^{|S|}=\sum_{i=0}^{r}\binom{r+1}{i}(r+1-i)^{|S|}(-1)^{i}.
\]
\end{proof}

\begin{proof}[\bf Proof of Theorem~\ref{l-bound}]
For all $r \ge k\ge 2$ and $S\subseteq [n]$ with $|S|\ge k$,
\begin{align*}
	g_{r,k}(S) \ge f^{<k}(r,S) &=  \sum_{j=1}^{k-1}  \binom{r}{j}f(j,S)= \sum_{j=1}^{k-1} \bigg( \binom{r}{j} \sum_{i=0}^{j-1}\binom{j}{i}(j-i)^{|S|}(-1)^{i} \bigg)\\	
	&\xlongequal[]{t:=j-i} \sum_{j=1}^{k-1} \bigg( \binom{r}{j} \sum_{j-t=0}^{j-t=j-1}\binom{j}{j-t}t^{|S|}(-1)^{j-t}\bigg)= \sum_{j=1}^{k-1} \bigg( \binom{r}{j} \sum_{t=1}^{j}\binom{j}{j-t}t^{|S|}(-1)^{j-t}\bigg).   
\end{align*}
Note that the constraints are $1\le t \le j \le k-1$ in the last formula. 
Therefore, we can change the order of summation and obtain 
\begin{align*}
	g_{r,k}(S) \ge \sum_{t=1}^{k-1}\bigg(t^{|S|}\sum_{j=t}^{k-1}\binom{r}{j}\binom{j}{j-t}(-1)^{j-t}\bigg).   
\end{align*}

\end{proof}

\section{Proof of Theorem~\ref{up-bound-dense}}\label{pf-u-bound}
In this section, we shall prove Theorem~\ref{up-bound-dense} using Corollary~\ref{r-k-ap-container}. Throughout this section, we restrict ourselves to sufficiently large $n$ and subset $S\subseteq[n]$ with $|S|=(1-\xi)n$ and $0\leq \xi \leq r^{-3}/2$.

Recalling the definition of $r$-template in Definition~\ref{def-template}, for a set $S\subseteq [n]$ and a collection $\mathcal{P}$ of $r$-templates of order $n$, denote by $G(\mathcal{P},S)$ the set of rainbow $k$-AP-free $r$-colorings of $S$, which is a subtemplate of some $P\in \mathcal{P}$. Let $g(\mathcal{P},S)=|G(\mathcal{P},S)|$.
If $\mathcal{P}$ consists of a single $r$-template $P$, then we simply write $G(P,S)$ and $g(P,S)$.

Let $\mathcal{C}$ be the collection of $r$-templates of order $n$ given by Corollary~\ref{r-k-ap-container}. By property (i) of Corollary~\ref{r-k-ap-container},  since every rainbow $k$-AP-free $r$-coloring of $S$ can be regarded as a subtemplate of some $P \in  \mathcal{C}$, we have $g_{r,k}(S)=g(\mathcal{C},S)$.
So it remains to estimate the upper bound of $g(\mathcal{C},S)$.
Set $\delta=\frac{1-\log(k-2)\cdot \log^{-1}(k-1)}{6k^3\log n}$.
Given a subset $S\subseteq[n]$, we divide $\mathcal{C}$ into two classes
\begin{equation}\label{def:coll}
\mathcal{C}_1=\{P\in \mathcal{C} : g(P,S)\leq (k-1)^{(1-\delta)n}\},
\quad
\mathcal{C}_2=\{P\in \mathcal{C} : g(P,S)> (k-1)^{(1-\delta)n}\}.
\end{equation}
The crucial part of the proof is to estimate the upper bound of $g(\mathcal{C}_2,S)$, which relies on the following three lemmas.

Note that every $P\in\mathcal{C}_2$ is a good $r$-template for $S$ by property (ii) of Corollary~\ref{r-k-ap-container}. For a template $P\in\mathcal{C}_2$ and a subset $S$, let
\[
X_1=\{x\in S: |P(x)|\leq k-2\},\quad \ X_2=\{x\in S : |P(x)|=k-1\},\quad \ X_3=\{x\in S:  |P(x)|\geq k\},
\] and $x_i=|X_i|$ for $i\in[3]$. Then by the definitions of $G(P,S)$ and $P\in \mathcal{C}_2$, we have 
\begin{equation}\label{eq:x3lb}
	(k-1)^{(1-\delta)n}<g(P,S)\le (k-2)^{x_1}(k-1)^{x_2}r^{x_3}.
\end{equation}

\begin{lemma} \label{x_3}
Let $r\ge k\ge 3$ and $\delta=\frac{1-\log(k-2)\cdot \log^{-1}(k-1)}{6k^3\log n}$. For every subset $S\subseteq[n]$ with $|S|=(1-\xi)n$ and $0\le \xi\le r^{-3}/2$, if there exists $P\in \mathcal{C}_2$, then $x_3 < 2kn^{\frac{k-1}{k}}$ and $\xi< \delta+2k(\log{r}-1)n^{-1/k}$.
\end{lemma}

\begin{proof}
Given $S\subseteq [n]$, let $P\in \mathcal{C}_2$. We first claim that $x_3=o(n)$. Otherwise, by Theorem~\ref{supersaturation}, we immediately have $\Gamma_k(X_3)=\Omega(n^2)$.
For every $k$-AP $K\subset X_3$, due to $|P(x_i)|\ge k$ for all $x_i\in K$, there exists a rainbow $k$-AP subtemplate $P'\subseteq P$ such that $|P'(x)|=1$ for $x\in K$ and $|P'(y)|=0$ for $y\in [n]\setminus K$.
Therefore, $R_k(P)=\Omega(n^2)$, which contradicts the condition that $P$ is a good $r$-template.

Note that  $\sum_i x_i=|S|=(1-\xi)n$. By Inequality (\ref{eq:x3lb}), we get
\[
x_3>\frac {(\xi- \delta)n \log(k-1)+x_1(\log(k-1)-\log(k-2))}{\log{r}-\log(k-1)},
\]
which also implies that $x_1=o(n)$. Let $\bar {X_2}=[n]\setminus X_2$, then $|\bar {X_2}|=x_1+x_3+\xi n<r^{-3} n$.

Next, we consider the family $\mathcal{K}$ of $k$-APs defined as follows
\[
\mathcal{K}=\{K\subset X_2\cup X_3 : |K\cap X_2|=k-1 \ \text{and} \  |K\cap X_3|=1\}.
\]
Note that every $k$-AP in $\mathcal{K}$ corresponds to at least one rainbow $k$-AP subtemplate of $P$, so we have $|\mathcal{K}|\le R_k(P)$.
In order to estimate the lower bound of $|\mathcal{K}|$, we consider an iterative algorithm (Algorithm~\ref{algorithm 1}). Since this algorithm will also be used in the following proof later, we consider the general input: an initial term $a\in [n]$, a common difference set $D\subset \mathbb{Z}$ and a underlying ground set $B\subseteq [n]$.
\begin{algorithm}[ht]
\caption{ (Iterative Algorithm) }
\label{algorithm 1}
{\bf Input: } $a$, $D$, $B$
\begin{algorithmic}[1]
  \For{$i = 1$ to $k - 1$}
    \State $D' = D\cap \{d: a + id \in B\}$
    \State $D = D'$
  \EndFor
  \State \Return $D^*=D$
\end{algorithmic}
\end{algorithm}

Using Algorithm~\ref{algorithm 1}, we input $a\in X_3$ and $B=X_2$.
If $a<n/2$, then we take $D=\{d\in \mathbb{Z}: 0<d<\frac{n}{2(k-1)}\}$; otherwise, $D=\{d\in \mathbb{Z}: -\frac{n}{2(k-1)}<d<0\}$. In each iteration of Algorithm~\ref{algorithm 1}, it is easy to check that the size of $D$ is reduced by at most $|\bar {X_2}|$. Therefore, for every $a\in X_3$, the final $D^*$ satisfies $|D^*|\ge \frac{n}{2(k-1)}-(k-1)|\bar {X_2}|>\frac{n}{k^2}$, which implies $|\mathcal{K}|
\ge \frac{x_3 n }{2k^2}$. Since $P$ is a good $r$-template of $S$, we have
\[
k^{-1}n^{2-1/k}> R_k(P) > \frac{n x_3}{2k^2},
\]
which indicates $x_3< 2kn^{1-1/k}$.
Moreover, by Inequality (\ref{eq:x3lb})
we have $\xi< \delta+2k(\log{r}-1)n^{-1/k}$ since $\sum_i x_i=|S|=(1-\xi)n$.
\end{proof}

Next, for $P\in \mathcal{C}_2$, we will describe the properties of $P$ through the following lemma.

\begin{lemma} \label{k-1 palette number}
Let $r\ge k\ge 3$, $\delta=\frac{1-\log(k-2)\cdot \log^{-1}(k-1)}{6k^3\log n}$ and $0\leq \xi< \delta+2k(\log{r}-1)n^{-1/k}$. Then for $S\subseteq [n]$ with $|S|=(1-\xi)n$ and $P\in \mathcal{C}_2$, there exist $k-1$ colors $\{c_1,\dots,c_{k-1}\}\subset [r]$ such that the number of integers in $S$ with palette $\{c_1,\dots,c_{k-1}\}$ is more than $(1-1/(3k^3\log n))n$.
\end{lemma}

\begin{proof}
Note that  $\sum_i x_i=|S|=(1-\xi)n$. By Inequality (\ref{eq:x3lb}), we get
\begin{align*}
	x_1<\frac {x_3(\log r-\log(k-1))-(\xi- \delta)n \log(k-1)}{\log(k-1)-\log(k-2)}.
\end{align*}
Moreover, by Lemma~\ref{x_3} we have $x_3< 2kn^{1-1/k}$. By combining the two inequalities,  we get
\begin{equation}\label{eq:x2lb}
x_2=(1-\xi)n-(x_1+x_3) >
(1-\frac{\log(k-1)}{\log(k-1)-\log(k-2)}\delta-\frac{\log r-\log(k-2)}{\log(k-1)-\log(k-2)} 2k n^{-1/k})n.
\end{equation}
Recalling that $\delta=\frac{1-\log(k-2)\cdot \log^{-1}(k-1)}{6k^3\log n}$, we have $|\bar X_2|=n-x_2=o(n)$. For every $C\subset [r]$ with $|C|=k-1$, let $Y_{C}=\{x\in X_2 : P(x)=C\}$.
By the pigeonhole principle, there exists a $C_0\subset [r]$ with $|C_0|=k-1$ such that
$|Y_{C_0}|\ge x_2/{\binom{r}{k-1}}\ge r^{1-k}n$. Let $Y'=X_2\setminus Y_{C_0}$, then we have the following claim.
\begin{claim}
$|Y'|=o(n)$.
\end{claim}
\begin{proof}
Assume for the sake of contradiction that $|Y'|=\Omega(n)$. In this case, it is sufficient to show that $R_k(P)=\Omega(n^2)$, which contradicts the condition that $P$ is a good $r$-template for $S$.

Let us consider a partition of $X_2$ into $V_1,V_2, \dots, V_{6k}$ satisfies $||V_i|-|V_j||\le 1$ and $\max (V_i)<\min (V_j)$ for $1\le i<j\le 6k$. 
Then there is an $j_0\in [6k]$ such that either $|V_{j_0}\cap Y_{C_0}| = \Omega(n)$ and $|V_{j_0+1}\cap Y'| = \Omega(n)$ or
the other way around.
Indeed, since $|Y_{C_0}|=\Omega(n)$, $|Y'|=\Omega(n)$ and $Y_{C_0}\cup Y'=X_2$, either $|V_j \cap Y_{C_0}|=\Omega(n)$ or $|V_j \cap Y'|=\Omega(n)$ for every $j\in [6k]$.
For convenience, set $Y_1=Y_{C_0}\cap (V_{j_0}\cup V_{j_o+1})$ and $Y_2=Y'\cap (V_{j_0}\cup V_{j_o+1})$.
For any $a\in Y_1$ and $b\in Y_2$, if $\max\{a,b\}<n/2$, let $K_{a,b}$ denote the $k$-AP in $[n]$ with $\min\{a,b\}$ as the first term and $|a-b|$ as the common difference; otherwise, let $K_{a,b}$ denote the $k$-AP in $[n]$ with $\max\{a,b\}$ as the $k$th term and $|a-b|$ as the common difference.
Note that for such $K_{a,b}$, we have $P(a)=C_0$,  $P(b)\neq C_0$ and $|P(a)|=|P(b)|=k-1$.
If $K_{a,b}\subset X_2$, then $P$ contains at least one rainbow $k$-AP subtemplate $P'\subseteq P$ such that $|P'(x)|=1$ for $x\in K_{a,b}$ and $P'(y)=\emptyset$ for $y\in [n]\setminus K_{a,b}$ since $|\cup_{x\in  K_{a,b}} P(x)|\ge k$.
Therefore, we will count the number of $K_{a,b}$ in $X_2$.

Given $a_0\in Y_1$, let $\mathcal{K}_{a_0}=\{K_{a_0,b}\subset [n]: b\in Y_2\}$. Due to $|a_0-b|<|V_{j_0}\cup V_{j_o+1}|+n-x_2<\frac{n}{2(k-1)}$ for any $b\in Y_2$, we have $|\mathcal{K}_{a_0}|=|Y_2|$. For every family $\mathcal{K}_{a_0}$ we consider
$\mathcal{K}^{*}_{a_0}=\{K_{a_0,b}\in \mathcal{K}_{a_0}:  K_{a_0,b}\subset X_2\}$.
Recalling the definition of $K_{a,b}$, for any $x\in \bar X_2$, $x$ is contained in at most $4(k-2)$ copies of $K_{a_0,b}$ for some $b\in Y_2$,
so $|\mathcal{K}^{*}_{a_0}|\ge |Y_2|-4(k-2)|\bar X_2|$.
Let $\mathcal{K}^{*}=\cup_{a_0\in Y_1}\mathcal{K}^{*}_{a_0}$. Since $|\bar X_2|=o(n)$, $|Y_1|=\Omega(n)$ and $|Y_2|=\Omega(n)$, we have
\[
|\mathcal{K}^{*}|\ge \left(|Y_2|-4(k-2)|\bar X_2|\right)\frac{|Y_1|}{2}=\Omega(n^2),
\]
which yields $R_k(P)\ge |\mathcal{K}^{*}|=\Omega(n^2)$.
\end{proof}
Next, we will count the number of $k$-APs in $X_2$ such that the 1st term (or $k$th term) is in $Y'$ and the remaining items are in ${Y}_{C_0}$. Similar to the analysis of $K_{a,b}$,  every such $k$-AP corresponds to at least one rainbow $k$-AP subtemplate of $P$.
Using Algorithm~\ref{algorithm 1}, we input $a\in Y'$ and $B=Y_{C_0}$. If $a<n/2$, we take $D=\{d\in \mathbb{Z}: 0<d<\frac{n}{2k-2}\}$; otherwise, $D=\{d\in \mathbb{Z}: -\frac{n}{2k-2}<d<0\}$.
Since $|\bar X_2|=o(n)$ and $|Y'|=o(n)$, we have $|[n]\setminus {Y}_{C_0}|=|\bar X_2|+|Y'|=o(n)$. Similar to Lemma~\ref{x_3}, for every $a\in Y'$, the final $D^*$ satisfies $|D^*|>\frac{n}{k^2}$, which implies that there are at least $\frac{|Y'|n}{2k^2}$ $k$-APs $K$ such that $|K\cap Y'|=1$ and $|K\cap Y_{C_0}|=k-1$.
Since $P$ is a good $r$-template of $S$, we have
\[
k^{-1}n^{2-1/k}> R_k(P) > \frac{|Y'|n}{2 k^2},
\]
which indicates $|Y'|< 2kn^{1-1/k}$.
By (\ref{eq:x2lb}) and $\delta=\frac{1-\log(k-2)\cdot \log^{-1}(k-1)}{6k^3\log n}$, we have $|Y_{C_0}|=x_2-|Y'|> (1-1/(3k^3\log n))n$, which completes the proof.
\end{proof}
\smallskip

For $S\subseteq [n]$ and $C\subset [r]$ with $|C|=k-1$, we define
\[
\mathcal{P}_C:=\{P\in \mathcal{C}_2: |\{x\in S: P(x)=C\}|> (1-\lambda)n\}, \ \text{where} \ \lambda=1/(3k^3\log n).
\]
According to Lemma~\ref{k-1 palette number}, it is easy to check that $\{\mathcal{P}_C : C\in \binom{[r]}{k-1}\}$ forms a partition of $\mathcal{C}_2$. For every $C\subset [r]$ with $|C|=k-1$, we will give an upper bound of $g(\mathcal{P}_C, S)$.

\begin{lemma} \label{P-k-1-bound}
Let $r\ge k\ge 3$, $\delta=\frac{1-\log(k-2)\cdot \log^{-1}(k-1)}{6k^3\log n}$ and $0\leq \xi< \delta+2k(\log{r}-1)n^{-1/k}$. Then for $S\subseteq [n]$ with $|S|=(1-\xi)n$ and $C\subset [r]$ with $|C|=k-1$, we have
\[
g(\mathcal{P}_C, S)< (k-1)^{|S|} + (k-1)^{|S|- \frac{n}{3k^3}}.
\]
\end{lemma}

\begin{proof}
For any rainbow $k$-AP-free $r$-coloring $\chi \in G(\mathcal{P}_C, S)$, let $S_\chi=\{x\in S: \chi(x)\notin C\}$.
By the definition of $\mathcal{P}_C$, we have $|S_\chi|< \lambda n$.
Let
\[\mathcal{G}_0=\{\chi \in G(\mathcal{P}_C, S) : S_\chi=\emptyset\}
 \quad\text{and}\quad
\mathcal{G}_1=\{\chi \in G(\mathcal{P}_C, S) : S_\chi\neq \emptyset\}.\]
Clearly, $g(\mathcal{P}_C, S)=|\mathcal{G}_0| + |\mathcal{G}_1|\le (k-1)^{|S|}+|\mathcal{G}_1|$. It remains to show that $|\mathcal{G}_1|<(k-1)^{|S| - \frac{n}{3k^3}}$.

Let us consider the possible $r$-colorings in $\mathcal{G}_1$. We first choose a subset $S_1\subseteq S$ with $0<|S_1|<\lambda n$, and use colors in
$[r]\setminus C$ to color a fixed $S_1$. Then the number of choices of $S_1$ is at most $\sum_{1\leq \ell < \lambda n}\binom{n}{\ell}$, and the number of colorings is at most  $(r-k+1)^{|S_1|}$.
After we fix a subset $S_1$ and its coloring, we consider the number of ways to color $S\setminus S_1$ using colors in $C$. Set $S_2= S\setminus S_1$.

For any fixed $a'\in S_1$, if $a'<n/2$, let $K_{a'}$ denote the $k$-AP in $S_2\cup \{a'\}$ with ${a'}$ as the 1st term, otherwise, let $K_{a'}$ denote the $k$-AP in $S_2\cup \{a'\}$ with $a'$ as the $k$th term. By symmetry, let $a'<n/2$ and $\mathcal K_{a'}$ denote the family of such $K_{a'}$.
Using Algorithm~\ref{algorithm 1}, we take $a=a'$, $D=\{d\in \mathbb{Z}: 0<d<\frac{n}{2(k-1)}\}$ and $B=S_2$. Since $\delta= \frac{1-\log(k-2)\cdot \log^{-1}(k-1)}{6k^3\log n}$ and $\xi< \delta+2k(\log{r}-1)n^{-1/k}$, we have $|[n]\setminus S_2|=n-|S|+|S_1|<2\lambda n$. Recalling Algorithm~\ref{algorithm 1}, we have 
\[
|\mathcal K_{a'}|\ge \frac{n}{2(k-1)}-(k-1)\cdot|[n]\setminus S_2|>\frac{n}{2(k-1)}-2(k-1)\lambda  n>\frac{n}{2k-1},
\]
where the last inequality follows from $\lambda=1/(3k^3\log n)$.

Next fix a $K\in \mathcal K_{a'}$, then the number of $K'\in \mathcal K_{a'}$ with $|K\cap K'|\ge 2$ is at most $(k-1)(k-2)+1 =k^2-3k+3$, since the $i$th term of $K$ with $2\le i\le k$ may be the $j$th term of $K'$ with $K'\neq K$ and $j\in [k]\setminus \{1,i\}$.  
Let $K^-_{a'}=K_{a'}\setminus \{a'\}$. Then the number of disjoint $K^-_{a'}$ is at least
\[
\frac{n}{2k-1}\cdot\frac{1}{k^2-3k+3}>\frac{1}{2k^3}n.
\]
In addition, for a $K\in \mathcal K_{a'}$ and a rainbow $k$-AP-free $r$-coloring $\chi$ of $S$ in $\mathcal{G}_1$, due to  $\chi(a')\notin C$ and $\chi(x)\in C$ for $x \in K\setminus \{a'\}$, there must be two integers $x_1, x_2\in K$ such that $\chi(x_1)=\chi(x_2)$.
Therefore, the number of ways to color $S_2$ is at most $(k-1)^{|S|-(1/2k^3) n}$.
Hence, we obtain that
\[
|\mathcal{G}_1|\leq
\sum_{1\leq \ell < \lambda n}\binom{n}{\ell} (r-k+1)^{\lambda n} (k-1)^{|S|-(1/2k^3) n}
< (k-1)^{\lambda n\frac{\log n+ \log r}{\log (k-1)}+ |S|-(1/2k^3) n} < (k-1)^{|S| - \frac{n}{3k^3}},
\]
where the last inequality follows from $\lambda=1/(3k^3\log n)$, which completes the proof.
\end{proof}

Now we have all the ingredients to give our proof of Theorem~\ref{up-bound-dense}.

\begin{proof}[{\bf Proof of Theorem~\ref{up-bound-dense}}]
Let $\mathcal{C}$ be the collection of $r$-templates of order $n$ given by Corollary~\ref{r-k-ap-container}, and $S\subseteq [n]$ with $|S|\ge (1-\xi)n$ and $0\le \xi \le r^{-3}/2$.
By Property (iii) of Corollary~\ref{r-k-ap-container} and the definition of $\mathcal{C}_1$ (see~(\ref{def:coll})), we have
\[
g(\mathcal{C}_1, S) \leq |\mathcal{C}_1|\cdot (k-1)^{(1-\delta)n} \leq | \mathcal{C}|\cdot (k-1)^{(1-\delta)n} <(k-1)^{n-\frac{6\delta n}{7}}=(k-1)^{n-\frac{1-\log(k-2)\cdot \log^{-1}(k-1)}{7k^3\log n}n}.
\]
If $\xi\ge  \delta+2k(\log{r}-1)n^{-1/k}$, by Lemma~\ref{x_3}, we are done by $g(\mathcal{C},S)=g(\mathcal{C}_1, S)$.
Otherwise, by Lemma \ref{k-1 palette number} and Lemma \ref{P-k-1-bound}, we have
\begin{align*}
g(\mathcal{C}, S) &= g(\mathcal{C}_{1}, S)+g(\mathcal{C}_{2}, S)\\
&\leq (k-1)^{n-\frac{6\delta n}{7}}+\binom{r}{k-1}(k-1)^{|S|}(1 + (k-1)^{- \frac{n}{3k^3}})\\
&\leq  \binom{r}{k-1}(k-1)^{|S|}+(k-1)^{-\frac{1-\log(k-2)\cdot \log^{-1}(k-1)}{8k^3\log n}n} (k-1)^n,
\end{align*}
which gives the desired upper bound on the number of rainbow $k$-AP-free $r$-colorings of $S$.
\end{proof}

\section{Proof of Theorem~\ref{r-extremal}}\label{prove-r-extremal}
In this section, we shall prove Theorem~\ref{r-extremal} using Theorem~\ref{l-bound}, Theorem~\ref{up-bound-dense} and Theorem~\ref{Sz-thm}.
%Throughout this section, we still fix $r$ a positive integer with $r\ge k\ge 3$.

\begin{proof}[{\bf Proof of Theorem~\ref{r-extremal}}]
By Theorem~\ref{l-bound}, we first have
\begin{equation}\label{g(n)lowbound}
g_{r,k}([n])\ge \sum_{t=1}^{k-1}\left(t^{n}\sum_{j=t}^{k-1}\binom{r}{j}\binom{j}{j-t}(-1)^{j-t}\right)\ge
\left(\binom{r}{k-1}-o(1)\right)\left(k-1 \right)^n.
\end{equation}
For any subset $S\subseteq [n]$ with $(1-r^{-3}/2)n\le |S|<n$, by Theorem~\ref{up-bound-dense}, we have
\[
g_{r,k}(S)\le \binom{r}{k-1}(k-1 )^{n-1}+(k-1)^{(1-\frac{1-\log(k-2)\cdot \log^{-1}(k-1)}{8k^3\log n})n}\le \left(\frac{\binom{r}{k-1}}{k-1}+o(1)\right)(k-1)^{n}< g_{r,k}([n]).
\]
For any subset $S\subseteq [n]$ with $|S|<\frac{n}{\log r}$, we have
\[
g_{r,k}(S)\le r^{|S|}\le (k-1)^{|S|\log r}<(k-1)^n< g_{r,k}([n]).
\]

Next, suppose that $|S|=\gamma n$ with $\log^{-1} r<\gamma< 1-r^{-3}/2$.
Let integer $q = n_0$, where $n_0$ is obtained by using Theorem~\ref{up-bound-dense}. Recalling Theorem~\ref{Sz-thm}, we take $\alpha= \frac{1}{3r^3\log r}$ and $n\ge sz(q, \alpha)$. Then $S$ contains at least $\frac{(\gamma - \alpha)n}{q}$ pairwise disjoint $q$-APs. For any $q$-AP $Q=\{a, \dots, a+(q-1)d\}$, let $\varphi$ be a mapping from $[q]$ to $Q$ such that $\varphi(i)= a+ (i-1)d$ for every $i \in [q]$. Then there is a bijection $\phi$ between all the
$k$-APs in $[q]$ and all the $k$-APs in $Q$ such that $\phi(\{a_1,\dots,a_k\}) = \{\varphi(a_1), \dots, \varphi(a_k)\}$ for any $k$-AP $\{a_1,\dots,a_k\}$ in $[q]$.
Therefore, $g_{r,k}(Q) = g_{r,k}([q])$. By the choice of $q$ and Theorem~\ref{up-bound-dense}, we have

\[
g_{r,k}(Q) = g_{r,k}([q])\le \binom{r}{k-1}(k-1)^{q}+(k-1)^{\big(1-\frac{1-\log(k-2)\cdot \log^{-1}(k-1)}{8k^3\log q}\big)q}< (k-1)^{q+(k-1)\log r},
\]
which yields 
\begin{align*}
g_{r,k}(S) &= r^{\alpha n}((k-1)^{q+(k-1)\log r})^{\frac{(\gamma-\alpha)n}{q}}
< (k-1)^{{\alpha n}\log r+(\gamma-\alpha)n+ \frac{(\gamma-\alpha)k\log r}{q}n}\\
&\le (k-1)^{ \gamma n+ \alpha n \log r+ \frac{\gamma k\log r}{q}n}
\le (k-1)^{ (1-\frac{1}{2r^3}) n+ \frac{1}{3r^3} n + \frac{\gamma k\log r}{n_0}n}\\
&<(k-1)^n < g_{r,k}([n]).
\end{align*}

Recalling Theorem~\ref{up-bound-dense}, we have
$g_{r,k}([n])\le (\binom{r}{k-1}+o(1))(k-1)^n$.
Combining the lower bound in (\ref{g(n)lowbound}), we can obtain immediately
\[
\lim_{n\to \infty} \frac{g_{r,k}([n])}{(k-1)^n}=\binom{r}{k-1}.
\]
%The proof is complete.
\end{proof}

%%%%%%%%%%%%%%%%%%%%%%%%%%%%%%%%%%%%%%%%%%%%%%%%%%%%%%%%%%%%%%%%%%%%%%%%%%%%%%%%%%%%%%%%%%
\section{Concluding Remarks}\label{remarks}
In this paper, we study rainbow Erd\H{o}s-Rothschild problem for $k$-AP-free set in $[n]$.
In fact, our results can be extended to cyclic groups.
%if we consider this problem on cyclic group, then we can get similar results.
Let $\mathbb{Z}_n$ be the \textit{cyclic group} of order $n$ formed by the set $\{0, 1, \dots , n-1\}$ with the binary operation addition modulo $n$.
For any $S\subseteq \mathbb{Z}_n$ and $r\ge k \ge 3$, let $g_{r,k}(S,\mathbb{Z}_n)$ denote the number of rainbow $k$-AP-free $r$-colorings of $S$.
Since arithmetic progressions can ``wrap around" in $\mathbb{Z}_n$, we immediately have the following corollary.

\begin{corollary}\label{Z_n}
For all integers $r \ge k\ge 3$, there exists $n_0\in \mathbb{N}^*$ such that for all $n\ge n_0$ and any proper subset $S\subset \mathbb{Z}_n$, we have $g_{r,k}(S,\mathbb{Z}_n )<g_{r,k}(\mathbb{Z}_n,\mathbb{Z}_n)$.
Moreover,
\[
\lim_{n\to \infty} \frac{g_{r,k}(\mathbb{Z}_n,\mathbb{Z}_n)}{(k-1)^n}=\binom{r}{k-1}.
\]
\end{corollary}

%When $k=3$ and $p$ is a prime number,
Li, Broersma and Wang~\cite{2021Integer} determined the exact value of $g_{r,3}(\mathbb{Z}_p,\mathbb{Z}_p)$ for prime $p$, using the result of $3\le aw(\mathbb{Z}_p, 3) \le 4$, which was proved in~\cite{anti-r-AP}.  However,
for $k\ge 4$, $ aw(\mathbb{Z}_n, k)= n^{1-o(1)}$ for all large $n$ (see~\cite{anti-r-AP}).
Therefore, it would be an interesting project to study the exact value of $g_{r,k}(\mathbb{Z}_n,\mathbb{Z}_n)$ for $k\ge 4$.

Another direction is that one can consider more general form of this problem.
Note that $k$-APs can be viewed as the set of distinct-valued solutions of the following homogeneous system of $k-2$ linear
equations
\begin{equation}
{\left( \begin{array}{ccccccccc}
1 & -2 & 1  & 0 & 0 &\cdots & 0 & 0 & 0\\
0 & 1  & -2 & 1 & 0 &\cdots & 0 & 0 & 0\\
  &\    &\    &\ &\      & \ddots \\
0 & 0  & 0 & 0 & 0 &\cdots & 1  & -2 & 1
\end{array}
\right )}
{\left(x_1,x_2,\dots, x_k \right )^T
}
{=\bf{0}.
}
\end{equation}
We quote the definition from~\cite{Ann-Schacht}. More generally, for an $\ell\times k$ integer matrix $M$ let $\mathcal{S}_0(M)\subseteq [n]^k$ be the set of solutions $(x_1,x_2,\dots, x_k)$ with all $x_i$ being distinct of the homogeneous system of linear equations given by $M$.
We say $M$ is \textit{irredundant} if $\mathcal{S}_0(M)\neq\emptyset$. Moreover, an irredundant $\ell\times k$ matrix $M$ is \textit{density regular}, if for every $\alpha>0$ there exists
an $n_0$ such that for all $n\ge n_0$ and $S\subseteq [n]$ with $|S|\ge \alpha n$ we have $S^k\cap \mathcal{S}_0(M)\neq\emptyset$. In~\cite{regular-matrix}, Frankl, Graham and R\"{o}dl characterized irredundant density regular matrices.

Given an $\ell \times k$ irredundant density regular integer matrix $M$ and a set $S\subseteq [n]$, an $r$-coloring of $S$ is called a \textit{rainbow $M$-free} if it contains no rainbow $k$-element set $K$ such that $K\in \mathcal{S}_0(M)$, where a $k$-element set  is called \textit{rainbow} if its each element is assigned different colors under the $r$-coloring of $S$.
In this context, the rainbow Erd\H{o}s-Rothschild problem for $M$-free set in $[n]$ is a very natural one.

\begin{question}
Given an $\ell \times k$ irredundant density regular integer matrix $M$ and a positive integer $r$, which subset of $[n]$ admits the maximum number of rainbow $M$-free $r$-colorings?
\end{question}

Furthermore, the method  can be used to prove Theorem~\ref{up-bound-dense} can also prove the analogous results for other irredundant density regular integer matrix. However, the stability analysis on other parts would be very involved. A well-known concept in additive number theory is the notion of a Sidon set. A set $S\subseteq [n]$ is called a \textit{Sidon set} if all the sums $x_1+x_2$ with $x_1<x_2$ and $x_1, x_2\in S$, are distinct. In 1941, Erd\H{o}s and  Tur\'an~\cite{E-T-Sidon} showed that the
maximum possible size $f(n)$ of a Sidon subset of $[n]$ satisfies $f(n)\le \sqrt{n} +O(n^{1/4})$.
Clearly, every $[1, -1, 1, -1 ]$-free set in $[n]$ is a Sidon set.
From the proof of Theorem~\ref{up-bound-dense}, we have the following conjecture.

\begin{conjecture}
For any integer $r\ge 2$ and sufficiently large $n$,  $[n]$ is always the unique subset admitting the maximum number of rainbow $[1, -1, 1, -1 ]$-free $r$-colorings among all subsets of $[n]$. In particular, almost all rainbow $[1, -1, 1, -1 ]$-free $r$-colorings of $[n]$ use at most $3$ colors.
\end{conjecture}

\section*{Acknowledgements}
The authors extend gratitude to the referees for their careful
peer reviewing and helpful comments.

\bibliographystyle{abbrv}
\bibliography{ref}

\end{document}